\documentclass[11pt]{amsart}

\usepackage{amssymb, amsthm, amsmath}
\usepackage{amsfonts}
\usepackage{graphicx, comment}
\usepackage[small]{caption}
\usepackage{subcaption}
\usepackage{epsfig}
\usepackage{tikz, float}
\usepackage [english]{babel}
\usepackage{tikz}
\usetikzlibrary{calc,decorations.pathreplacing}

\usepackage{complexity}

\newcommand{\later}[1]{}
\newcommand{\old}[1]{}

\usepackage[utf8]{inputenc}
\usepackage{fullpage}
\usepackage{framed}

\usepackage{enumerate}
\usepackage{url}
\usepackage[breaklinks]{hyperref}
\hypersetup{
	colorlinks = true, 
	urlcolor = cyan, 
	linkcolor = teal, 
	citecolor = cyan 
}

\newtheorem{theorem}{Theorem}[section]
\newtheorem{lemma}[theorem]{Lemma}

\newtheorem{definition}[theorem]{Definition}

\newtheorem{proposition}[theorem]{Proposition}
\newtheorem{corollary}[theorem]{Corollary}

\newcommand{\ie}{{i.e.}}
\newcommand{\eg}{{e.g.}}

\newcommand{\eps}{\varepsilon}

\newcommand{\ZZ}{\mathbb{Z}} 
\newcommand{\RR}{\mathbb{R}} 

\def\C{\mathcal C}

\def\S{\mathcal S}

\usepackage{tikz}
\tikzstyle{vtx} = [circle, fill, inner sep=0.7]

\title{Exponential improvements in Rado's covering problem}

\thanks{The first named author would like to acknowledge the financial support
  of the Istituto Nazionale di Alta Matematica ``F. Severi".
}

\author{Gian Maria Dall'Ara}
\address{Gian Maria Dall'Ara \newline
  Istituto Nazionale di Alta Matematica ``F. Severi''
  and Scuola Normale Superiore, Pisa, Italy}
\email{dallara@altamatematica.it}

\author{Adrian Dumitrescu}
\address{Adrian Dumitrescu \newline
  Algoresearch L.L.C., Milwaukee, WI, USA \newline
  Alfr\'ed R\'enyi Institute of Mathematics, Budapest, Hungary \newline
  Research Institute of the University of Bucharest, Romania}
\email{ad.dumitrescu@algoresearch.org}

\date{\today}

\begin{document}

\begin{abstract}
Let $B^d$ denote the $d$-dimensional Euclidean ball of unit radius. What is the largest constant $f(B^d) \in [0,1]$ with
the property that every finite collection $\mathcal{C}$ of unit balls in $\RR^d$ admits a disjoint sub-collection
$\mathcal{S}$ occupying at least a fraction $f(B^d)$ of the volume of $\mathcal{C}$? This problem was first raised by
T.~Radó in 1928, for axis-parallel squares in the plane.
The case of Euclidean balls was first considered by R. Rado in 1949. This paper focuses on the asymptotic
behavior of $f(B^d)$ in the high-dimensional regime $d\rightarrow \infty$.  

Until last year the best known estimates for unit balls were very far apart:
$(1+\epsilon_d) 3^{-d} \leq f(B^d) \leq 2^{-d}$,  for some explicit positive $\epsilon_d=o_{d\rightarrow \infty}(1)$.
Recently, the lower bound was improved by a linear factor by Xie and Ge,
so $f(B^d) = \Omega(d \cdot 3^{-d})$. 

Here we offer the first exponential improvements of both the upper and the lower bounds in almost 80 years,
which narrows the gap to:
\[ 2.910^{-d} \leq f(B^d) \leq 2.447^{-d} \]
for large $d$. The lower bound is constructive and yields a polynomial time algorithm (in the number of balls and the dimension)
for finding a disjoint sub-collection realizing the estimate. The same technique gives similar exponentially improved
lower bounds for all symmetric convex bodies satisfying a uniform convexity assumption, e.g., $\ell^p$-balls for all $p\in (1,\infty)$.
The exponential improvement of the upper bound follows from the celebrated Kabatiansky--Levenshtein upper bound
for spherical codes.
\end{abstract} 

\maketitle

\section{Introduction} \label{sec:intro}

This paper is devoted to the ``covering problem'' first introduced by T.~Radó in a 1928 letter to Sierpinski~\cite{Ra28}
and formulated for general convex bodies by R.~Rado in a 1949 paper~\cite{R49}. We begin by formulating the problem, and point
the reader to our recent survey~\cite{DD26av3} for more mathematical and historical details.

For a measurable subset $E$ of $\RR^d$, denote by $|E|$ its Lebesgue measure.
If $K$ is a symmetric convex body in $\RR^d$, let $F(K)$ denote the largest constant $c\geq 0$
for which the following statement holds true: any finite collection $\mathcal{C}$ of
\emph{homothetic copies} of $K$ admits a disjoint sub-collection $\mathcal{S}$ such that  
\[
|\cup_{K'\in \mathcal{S}}K'|\geq c|\cup_{K'\in \mathcal{C}}K'|. 
\]
Similarly, we define a constant $f(K)$ by requiring that the above statement holds for collections
$\mathcal{C}$ consisting of \emph{translates} of $K$. The problem is to determine, or more
realistically estimate, the value of the constants $F(K)$ and $f(K)$ for various choices of $K$.
As remarked in~\cite{DD26av3}, the problem of estimating $f(B^d)$ (where $B^d$ is the
unit Euclidean ball in $\mathbb{R}^d$) is a sort of ``relative'' sphere packing problem, where one
is constrained to choose the balls from a prescribed and completely arbitrary collection.  

The basic estimates 
\begin{equation}\label{eq:basic_estimates}
  3^{-d} \leq F(K)\leq f(K) \leq 2^{-d}\qquad \forall K
  \subseteq \mathbb{R}^d,
\end{equation} 
have been the best known bounds for a long time, except for minor improvements in the lower bound
(of the form $(1+\epsilon_d)3^{-d}$ with some explicit positive $\epsilon_d=o(1)$ in the asymptotic regime $d\rightarrow \infty$).
Recently, there has been progress in the case $K=Q^d$ (where $Q^d$ is the unit axis-parallel cube) and $K=B^d$.
In \cite{D25}, the first author closed the exponential gap for high-dimensional cubes, showing that\footnote{Throughout
  this paper, $\log{x}$ stands for the natural logarithm of $x$.}
\[
F(Q^d)\geq (e+o(1))^{-1}(d\log d)^{-1}2^{-d}. 
\]
This inequality was derived from the general comparison
\begin{equation}\label{eq:comparison}
F(K)\geq (e+o(1))^{-1}(d\log d)^{-1}f(K), 
\end{equation}
valid for any symmetric convex body, and the identity $f(Q^d)=2^{-d}$, which had been known since
the 1940s (see \cite{DD26av3} for details on the history). Concerning the unit ball, as shown in~\cite[Theorem 6.2]{DD26av3},
an application of a classical result of Rankin and Davenport--Hajos yields the exponentially stronger upper bound
\[ f(B^d)=O(2.414^{-d}). \]
From the other direction, Xie and Ge~\cite{XG26} used the ``weighted hardcore model'' to establish that
\[ f(B^d) =\Omega(d \cdot 3^{-d}),  \]
Here we offer the first exponential improvement of this lower bound.

\begin{theorem} [Euclidean exponentially improved lower bound]  \label{thm:exp_ball}
  Let $\rho_2:= \sqrt{5 + 2 \sqrt3} = 2.9093\ldots<3$. Then 
\begin{equation}
f(B^d)\geq (2^d+\rho_2^d)^{-1}\label{eq:exp_ball}
\end{equation} and \begin{equation}\label{eq:exp_ball_2}
F(B^d)\geq (e+o(1))^{-1}(d\log d)^{-1}(2^d+\rho_2^d)^{-1}.
\end{equation}
\end{theorem}
The main novel result is \eqref{eq:exp_ball}, while \eqref{eq:exp_ball_2} follows from it by inequality~\eqref{eq:comparison}. 
As explained below our method applies to a variety of symmetric convex bodies, yielding in particular analogous exponentially
improved lower bounds for $\ell^p$ balls for all $p\in (1,\infty)$. 

The quantitative result of Theorem~\ref{thm:exp_ball} carries over in the algorithmic front.

\begin{theorem} [Euclidean exponentially improved lower bound, algorithmic version] \label{thm:alg}
Given a finite collection $\C$ of balls (of arbitrary radii) in $\RR^d$, a disjoint sub-collection $\S \subseteq \C$ with
\begin{equation}   \label{eq:alg}
|\cup_{B\in \mathcal{S}}B|  \geq 2.910^{-d} |\cup_{B\in \mathcal{C}}B|,
\end{equation}
for any sufficiently large $d$, can be computed by a randomized algorithm in time that is polynomial in
both $n$ and $d$.
(The factor  $2.910^{-d}$ can be replaced by $c \cdot 2.910^{-d} $, where $c>0$ is a universal constant, 
for a bound that holds in any dimension.)
\end{theorem}

The proof of Theorem \ref{thm:exp_ball} is obtained in Section \ref{sec:proof}, specializing a theorem
that applies to more general convex bodies, discussed in Section \ref{sec:general}.
The algorithmic result is in Section \ref{sec:alg}. 

From the other direction we prove:

\begin{theorem}[Kabatiansky--Levenshtein upper bound for $f(B^d)$]  \label{thm:upper}
\begin{equation} \label{eq:b>2.44}
  f(B^d) = O(2.447^{-d}).
\end{equation}
\end{theorem}

The proof of Theorem \ref{thm:upper} is obtained in Section~\ref{sec:upper}.

\subsection{Statement on AI usage}
The authors explored a strategy to improve the lower bound on $f(B^d)$ based on the decomposition of Section
\ref{sec:decomposition} below. Partial results in this direction (insufficient to close the argument) were recorded in
Section 7 of their preprint~\cite{DD26av2}. An LLM, prompted by the first author (on July 12th 2026), proved Theorem
\ref{thm:exp_ball}: it improved the ``independence number 2'' analysis of \cite{DD26av2}, observing that the set inclusion
$B\subseteq \sqrt{5+2\sqrt{3}}B_1\cup \sqrt{5+2\sqrt{3}}B_2$
is valid for any unit ball $B$ intersecting two disjoint unit balls $B_1$ and
$B_2$. This simple, yet key, observation has been developed by the authors into the present
paper. No LLM was involved in the writing of this manuscript.

\section{A general theorem}\label{sec:general}

\subsection{Bodies and norms} A symmetric convex body in $\mathbb{R}^d$ is the unit ball
\[
K=\{x\in \mathbb{R}^d\colon\, \lVert x\rVert_K\leq 1\}
\]
of a unique norm $\lVert \cdot \rVert_K:\mathbb{R}^d\rightarrow [0, \infty)$; homothetic copies of
  $K$ can then be thought of as metric balls in the normed space $(\mathbb{R}^d, \lVert\cdot,
  \rVert_K)$. The correspondence bodies--norms will be used freely below to rephrase statements
  about collections of homothetic or translated copies of $K$ as geometric-combinatorial statements
  about collections of points on a finite-dimensional normed space. E.g., the constant $f(K)$ is the
  largest $c\geq 0$ with the following property: every finite set $X\subseteq \mathbb{R}^d$ has a
  subset $Y\subseteq X$ such that
\begin{equation}\label{eq:2-sep}
\lVert y-y'\rVert_K>2\qquad \forall y,y'\in Y
\end{equation} and \begin{equation}\label{eq:density}
|\cup_{y\in Y}B(y,1)|\geq c|\cup_{x\in X}B(x,1)|, 
\end{equation}
where $B(x,r)$ is the ball of center $x$ and radius $r$ with respect to the norm $\lVert \cdot\rVert_K$.
In graph-theoretic terminology, condition \eqref{eq:2-sep} says that $Y$ is an independent set in
the
``proximity graph'' with vertex set $X$ and adjacency relation
\[x\sim x'\quad \Longleftrightarrow \quad \lVert x-x'\rVert_K\leq 2, \]
which may be identified with the intersection graph of the collection of unit balls specified by the
set of centers $X$. In what follows, this graph-theoretic terminology is used without
comment. Notice that the density bound \eqref{eq:density} shows that the question does not reduce to
a purely combinatorial problem on the proximity graph (hence the need for a ``weighted'' analysis,
as in~\cite{XG26}).  

\subsection{Maximum independent sets}

In view of the above, we now fix a norm $\lVert \cdot \rVert$
in $\mathbb{R}^d$ and work with the formulation of the problem just discussed. Let
$B(x,r)=\{y\in \mathbb{R}^d\colon\, \lVert y-x\rVert\leq r\}$ and $K=B(0,1)$. We begin by recalling
the proof of the easy lower bound $f(K)\geq 3^{-d}$.  

Given a finite set $X\subseteq \mathbb{R}^d$, let $Y$ be \emph{any maximal (w.r.t.~inclusion)
  independent set} of the proximity graph associated to $X$. This means that every $x\in X\setminus
Y$ is at distance at most $2$ from some element of $Y$. Hence,
\[\cup_{x\in X}B(x,1)\subseteq \cup_{y\in Y}B(y,3).\] 
This statement may be viewed as the equal-radius case of the ``Vitali covering lemma'';
see, e.g.,~\cite[Lemma 7.3]{R87}. Because the unit balls centered at Y are disjoint and dilating a
set by a factor $\lambda$ scales its volume by a factor at most $\lambda^d$,    
\begin{equation}\label{eq:vitali_bound} \Big| \bigcup_{x \in X} B(x,1) \Big| \leq 3^d\Big|
  \bigcup_{y \in Y} B(y,1) \Big|.
\end{equation} 
This is the desired $3^{-d}$ lower bound. Notice that we proved more than needed, namely that
\emph{every maximal independent set $Y$ satisfies \eqref{eq:vitali_bound}}. It is important to
notice that the latter conclusion cannot be sharpened: e.g., if $X=(\epsilon \ZZ)^d\cap B(0,2)$, then
$\bigcup_{x \in X} B(x,1)$ approximately fills $B(0,3)$ and the independent sets are just the
singletons, so $\frac{\Big| \bigcup_{y \in Y} B(y,1)\Big |}{\Big| \bigcup_{x \in X} B(x,1)\Big  |}\rightarrow 3^{-d}$
asymptotically as $\epsilon$ tends to zero. Thus, any improvement of~\eqref{eq:vitali_bound}
requires a more careful choice of $Y$.  

There is an obvious best choice for $Y$, namely \emph{any independent set of maximal
  cardinality}. If $\alpha$ is such cardinality, then by disjointness
$\Big| \bigcup_{y \in Y} B(y,1) \Big|=\alpha|K|$.
The statement that $f(K)\geq c'$ is \emph{equivalent} to the inequality
$\alpha|K|\geq c' \Big | \cup_{x\in X}B(x,1)\Big |$. Thus, in order to prove our main theorem we
need to exploit some property of maximum independent sets (in the sense of cardinality) and not
merely their maximality with respect to inclusion. 

\subsection{A decomposition of the set of centers}\label{sec:decomposition}

Let $X\subseteq \mathbb{R}^d$ be finite and
$Y\subseteq X$ a \emph{maximum} independent set. Denote by $\alpha$ its cardinality. In particular, $Y$ is
maximal and thus any $x\in X$ is at distance at most $2$ from some element of $Y$.  For $x \in X$,
define 
\[ N(x):= \{ y \in Y \ \colon \ \lVert x-y\rVert \leq 2\}. \]
Thus $N(x) \neq \emptyset$, for every $x \in X$. Partition $X$ into
\[ X_1:= \{ x \in X \ \colon \ \# N(x)=1\}, \qquad
X_{\geq 2}:= \{ x \in X \ \colon \ \# N(x) \geq 2\}.\]
We now treat the two sets of centers $X_1$ and $X_{\geq 2}$ in two different ways. We call the first
the \emph{singly-dominated centers}, and the second the \emph{multiply-dominated centers}. 

\subsection{Singly-dominated centers}\label{sec:singly_dominated}

For $X_1$, we use the \emph{isodiametric inequality}, which states that, among all compact sets
$E\subseteq \mathbb{R}^d$ of diameter at most $r$, the ball of radius $\frac{r}{2}$ has maximum
volume. Notice that here diameter and balls are defined with respect to the fixed norm $\lVert \cdot
\rVert$, and the result holds for any norm; see, e.g., the one-line proof using the Brunn--Minkowski
inequality~\cite[p.~93]{BZ88}.  

Clearly, \[X_1=\cup_{y\in Y}P_y,\qquad P_y:= \{x \in X \ \colon \ N(x) =\{y\}\}.\] 
The key observation is that each $P_y$ is a clique in the proximity graph, that is, any two points
of $P_y$ are at distance at most $2$. Indeed, if there were 
$x,x' \in P_y$ such that $|x-x'|>2$, then $Y'=\left(Y \setminus \{y\}\right) \cup \{x,x'\}$
would be independent, since both $x$ and $x'$ are farther than $2$ from every element of
$Y \setminus \{y\}$, by the definition of $P_y$. The cardinality of $Y'$ would be $\alpha +1$, 
contradicting the maximality of $|Y|$. Consequently, the set $A_y := \bigcup_{x \in P_y} B(x,1)$ has
diameter at most $4$ and, by the isodiametric inequality, its volume is at most $2^d|K|$.  
Hence
\begin{equation} \label{eq:singly_dominated}
	\big| \bigcup_{x \in X_1} B(x,1) \big| \leq \sum_{y \in Y} |\bigcup_{x \in P_y} B(x,1)| \leq \alpha 2^d |K|. 
\end{equation}
This completes the analysis of the singly-dominated-centers.

\subsection{Multiply-dominated centers}

We introduce a definition tailored to the analysis of $X_{\geq 2}$. 

\begin{definition}\label{def:rho2}
Given a symmetric convex body $K$, let $\rho_2(K)$ be the smallest constant $r>0$ such that the
following holds: if $K_1$ and $K_2$ are translates of $K$ with disjoint interiors, and $K_3$ is a
further translate of $K$ intersecting both of them, then $K_3\subseteq rK_1 \cup rK_2$. Here $rK_i$
is the dilate of $K_i$ with respect to its center. 
\end{definition}
This is the $m=2$ case of a sequence of constants $\rho_m(K)$, whose definition is sketched in Section
\ref{sec:concluding} below (see Remark (2)).  

It is easy to rephrase the definition in terms of the norm $\lVert \cdot \rVert=\lVert \cdot
\rVert_K$:
\begin{equation}\label{eq:rho2}
\rho_2(K)=\max_{(x,y_1,y_2)\in R} \min \{\lVert x-y_1\rVert, \lVert x-y_2\rVert\}, 
\end{equation}
where $R\subseteq (\mathbb{R}^d)^3$ is defined by the conditions
\begin{equation}\label{eq:rho2_inequalities}
\lVert x\rVert\leq 1, \quad 
\lVert y_1\rVert,  \lVert y_2\rVert\leq 2, \quad \lVert y_1-y_2\rVert\geq 2. 
\end{equation}

It is now straightforward to plug the newly introduced constant $\rho_2(K)$ into our argument. If
$x\in X_{\geq 2}$, then there are at least two unit balls (a.k.a.~translates of $K$) centered at
points of $Y$ that intersect the unit ball centered at $x$. By definition of $\rho_2(K)$,
\[
 \bigcup_{x \in X_{\geq 2}} B(x,1) \subseteq \bigcup_{y\in Y} B(y, \rho_2(K))
\] and \begin{equation}\label{eq:multiply_dominated}
	\big| \bigcup_{x \in X_{\geq2}} B(x,1) \big| \leq \alpha \rho_2(K)^d |K|.
\end{equation}

\subsection{General theorem}

Combining \eqref{eq:singly_dominated}, \eqref{eq:multiply_dominated},
and \eqref{eq:comparison}, we obtain the following theorem. 

\begin{theorem}\label{thm:abstract}
	Let $K$ be any symmetric convex body in $\mathbb{R}^d$. Then \[f(K)\geq
        (2^d+\rho_2(K)^d)^{-1}\] and
        \[
	F(K)\geq (e+o(1))^{-1}(d\log d)^{-1}(2^d+\rho_2(K)^d)^{-1}.
	\] 
\end{theorem}

Of course, the theorem provides novel information on the constants $f(K)$ and $F(K)$ only for those
convex bodies for which we can establish that $\rho_2(K)<3$. This matter is treated in the next
section.

\section{The constant $\rho_2(K)$} \label{sec:proof}

\subsection{Computation in the Euclidean case}

In order to complete the proof of Theorem \ref{thm:exp_ball}, we need to compute $\rho_2(B^d)$. For later purposes,
it is convenient to introduce a slightly more general quantity (see Section \ref{sec:dvoretzky}).
For $\delta\leq 2$, let
\begin{equation}\label{eq:rho2_delta} 
\rho_2(K,\delta)=\max_{(x,y_1,y_2)\in R(\delta)} \min \{\lVert x-y_1\rVert_K, \lVert x-y_2\rVert_K\}, 
\end{equation}
where $R(\delta)\subseteq (\mathbb{R}^d)^3$ is defined by the
conditions
\begin{equation}\label{eq:rho2_delta_inequalities}
\lVert x\rVert_K\leq 1, \quad  
\lVert y_1\rVert_K,  \lVert y_2\rVert_K\leq 2, \quad \lVert y_1-y_2\rVert_K\geq 2\delta.  
\end{equation}
Thus, $\rho_2(K)=\rho_2(K,1)$. 

Let $\lVert x\rVert_2:=\left(\sum_{j=1}^dx_j^2\right)^\frac{1}{2}$ ($x\in \mathbb{R}^d$) be the
Euclidean norm and $\langle{x,y}\rangle$ the corresponding scalar product ($x,y\in \mathbb{R}^d$). 
Recall that, for any two vectors $x,y \in \RR^d$, we have the \emph{parallelogram identity}. 

\begin{equation} \label{eq:parallelogram}
	\lVert x+y\rVert_2^2 + \lVert x-y\rVert_2^2 = 2 \lVert x\rVert _2^2 + 2 \lVert y\rVert_2^2.  
\end{equation}

\begin{proposition} \label{prop:rho2_euclidean}
  For all $d\geq 2$ and $\delta\leq 2$, we have
  \begin{equation}\label{eq:rho2_euclidean}
	\rho_2(B^d,\delta)=\sqrt{5+2\sqrt{4-\delta^2}}.
  \end{equation}
  In particular,
\[ \rho_2(B^d)=\sqrt{5+2\sqrt{3}}. \]
\end{proposition}

\begin{proof}
Let $x,y_1,y_2$ as in \eqref{eq:rho2_inequalities}, where the norm is Euclidean. Put
\[ p=\frac{y_1+y_2}{2},  \qquad q=\frac{y_1-y_2}{2}. \] 
Then $\lVert q\rVert_2\geq \delta$ and, by the parallelogram identity~\eqref{eq:parallelogram}
\[ \lVert p\rVert_2^2 + \lVert q\rVert_2^2 = \frac{\lVert y_1\rVert_2^2 + \lVert y_2\rVert_2^2}{2} \leq 4. \]
It follows that $\lVert p\rVert\leq \sqrt{4-\delta^2}$.

\begin{figure}[htbp]
	\centering
	\begin{tikzpicture}[scale=1.8,
		every node/.style={font=\small}]
		
		\def\deltaval{1.2}
		\pgfmathsetmacro{\hval}{sqrt(4-\deltaval*\deltaval)}
		
		\coordinate (y1) at (-\deltaval,0);
		\coordinate (m)  at (0,0);
		\coordinate (y2) at (\deltaval,0);
		\coordinate (v)  at (0,\hval);
		\coordinate (x)  at (0,\hval+1);
		
		\draw[thick] (y1) -- (v) -- (y2) -- cycle;
		
		\draw[dashed] (m) -- (v);
		
		\draw[thick] (v) -- (x);
		
		\draw[dashed] (y1) -- (x);
		
		\fill (y1) circle (1.2pt);
		\fill (y2) circle (1.2pt);
		\fill (m)  circle (1.2pt);
		\fill (v)  circle (1.2pt);
		\fill (x)  circle (1.2pt);
		
		\node[below left]  at (y1) {$y_1$};
		\node[below right] at (y2) {$y_2$};
		\node[above]       at (x)  {$x$};
		
		\node[below right]  at ($(y1)!0.50!(v)$) {$2$};
		
		\node[right] at ($(v)!0.5!(x)$) {$1$};
		
		\draw[
		decorate,
		decoration={brace,mirror,amplitude=4pt}
		] (y1) -- (m)
		node[midway,below=6pt] {$\delta$};

		\node[right=-1pt] at ($(m)!0.30!(v)$)
		{$\sqrt{4-\delta^2}$};
		
		\node[
		left,
		align=center,
		fill=white,
		inner sep=1.5pt
		] at ($(y1)!0.56!(x)$)
		{$\sqrt{\,5+2\sqrt{4-\delta^2}\,}$};
		
		\draw ($(m)+(-0.13,0)$)
		-- ($(m)+(-0.13,0.13)$)
		-- ($(m)+(0,0.13)$);
		
	\end{tikzpicture}
\caption{}
\label{fig:rho2_euclidean}
\end{figure}

We compute the average of the squared distances of $x$ to $y_1$ and $y_2$ respectively, 
again using the parallelogram identity~\eqref{eq:parallelogram}:
	\begin{align*}
		\frac{\lVert x-y_1\rVert_2^2 + \lVert x-y_2\rVert_2^2}{2} &= \lVert x-p\rVert_2^2 +
                \lVert q\rVert_2^2 = \lVert x\rVert_2^2 + \lVert p\rVert_2^2 +  \lVert q\rVert_2^2 -
                2 \langle{x,p}\rangle\\  
		&\leq 1+4+ 2\lVert p\rVert_2 \leq 5+2\sqrt{4-\delta^2}
	\end{align*}
	
At least one of $\lVert x-y_1\rVert_2^2, \lVert x-y_2\rVert_2^2$ is no larger than their average, so
	\[ \min\{\lVert x-y_1\rVert_2,\lVert x-y_2\rVert_2\} \leq  \sqrt{5 +2\sqrt{4-\delta^2}}. \]
The maximum is achieved with $x,y_1,y_2$ coplanar, as shown in Figure \ref{fig:rho2_euclidean},
where the origin is the tip of the isosceles triangle. This proves~\eqref{eq:rho2_euclidean}.  
\end{proof}

Combining Theorem \ref{thm:abstract} and Proposition \ref{prop:rho2_euclidean}, we have a proof of Theorem \ref{thm:exp_ball}.

\subsection{Comparison with modulus of convexity}

Clearly, $\rho_2(K)\leq 3$ for any convex body. Moreover, $\rho_2(Q^d)=3$ for $d\geq 2$, where
\[  Q(x,r) =[x_1-r, x_1+r]\times \cdots \times [x_d-r, x_d+r],\qquad x=(x_1,\ldots, x_d)\in \RR^d, \, r>0, \] 
and $Q^d:=Q(0,1)$ is the $d$-dimensional cube, as may be seen by
taking $y_1=(-2,2,0')$, $y_2=(2,2,0')$, $x=(0,-1,0')$ in \eqref{eq:rho2}, \eqref{eq:rho2_inequalities} (here
$0'=(0,\ldots, 0)\in \mathbb{R}^{d-2}$). 
We now show that, as soon as the boundary of $K$ is not too flat, then $\rho_2(K)<3$. We recall the
notion of \emph{modulus of convexity} of a norm $\lVert \cdot \rVert=\lVert \cdot \rVert_K$. This is
the function  
\[ 
\delta_K(\epsilon)=\min\left\{1-\left\lVert \frac{x+y}{2}\right\rVert\colon\, \lVert x\rVert,\lVert y\rVert\leq 1,
\quad \lVert x-y\rVert \geq \epsilon\right\}, \qquad \epsilon\in(0,2]. 
\]
See \cite{LT79} for its role in the theory of Banach spaces. 

\begin{proposition}\label{prop:rho2_vs_convexity}
	\[
	\rho_2(K)\leq 3-2\delta_K(1/2). 
	\]
\end{proposition}

\begin{proof}
Let $x,y_1,y_2\in \mathbb{R}^d$ be such that $\lVert x\rVert\leq 1$, $\lVert y_1\rVert, \lVert y_2\rVert\leq 2$ and $\lVert y_1-y_2\rVert\geq 2$.
We have to show that
\[
\min\{\lVert x-y_1\rVert, \lVert x-y_2\rVert\}\leq 3-2\delta_K(1/2). 
\]
By the triangle inequality,
\[ 2\leq \lVert y_1-y_2\rVert\leq \lVert 2x+y_1\rVert+\lVert 2x+y_2\rVert. \]
Thus either $\lVert 2x+y_1\rVert\geq 1$ or $\lVert 2x+y_2\rVert\geq 1$, and we may assume
without loss of generality that the first option occurs. Since $x$ and $-y_1/2$ are in the
unit ball and their distance is at least $1/2$, the definition of modulus of convexity yields
 \[  \lVert x/2-y_1/4\rVert\leq 1-\delta_K(1/2). \]
	Thus, the triangle inequality gives
 \[ \min\{\lVert x-y_1\rVert, \lVert x-y_2\rVert\} \leq \lVert x-y_1\rVert\leq
        \lVert x-y_1/2\rVert+1\leq3-2\delta_K(1/2),  \]
as required. 
\end{proof}

Theorem \ref{thm:abstract} yields the following. 

\begin{corollary}\label{cor:f(K)_vs_convexity}
Let $K$ be a symmetric convex body in $\mathbb{R}^d$. Then \[
f(K)\geq (2^d+(3-2\delta_K(1/2))^d)^{-1}
\]	 and \[
F(K)\geq (e+o(1))^{-1}(d\log d)^{-1}(2^d+(3-2\delta_K(1/2))^d)^{-1}. 
\] 
\end{corollary}

For the unit ball $B^d_p$ with respect to the $\ell^p$ norm \[
\lVert x\rVert_p:=\left(\sum_{j=1}^d |x_j|^p\right)^\frac{1}{p}, \qquad x\in \mathbb{R}^d, 
\] 
the modulus of convexity is \[
\delta_{B^d_p}(\epsilon) = \begin{cases} 0\qquad &(p=1)\\
  	\text{unique $\delta\in [0,1]$ s.t. }(1-\delta+\epsilon/2)^p+|1-\delta-\epsilon/2|^p=2 \qquad &(1<p<2)\\
 	1-(1-(\epsilon/2)^p)^\frac{1}{p}\qquad &(2\leq p<\infty)\\
	0\qquad &(p=\infty)
	\end{cases}
\]
for any $d\geq 2$ (see~\cite{H56}). In particular, $\delta_{B^d_p}(1/2)=:\delta_p>0$ for all $p\in (1,\infty)$,
uniformly for $d\geq 2$. Thus, Corollary \ref{cor:f(K)_vs_convexity} gives an exponential improvement
on the basic lower bound.   

\medskip
\noindent \textbf{Remark.}
Using the above formula for $\eps=1/2$ and $p=2$ in Proposition~\ref{prop:rho2_vs_convexity} gives
\begin{align*}
\delta_{B_2^d}(1/2) &= 1 - \left( 1 - 4^{-2} \right)^{1/2} = 1 -\sqrt{\frac{15}{16}}, \\
\rho_2 (B_2^d) &\leq 3 - 2 \delta_{B_2^d}(1/2) = 2.936\ldots,
\end{align*}
that is, an upper bound on $\rho_2 (B_2^d)$  that is weaker than what we obtained in Proposition~\ref{prop:rho2_euclidean},
namely $\rho_2(B_2^d)=\sqrt{5+2\sqrt{3}} = 2.909\ldots$.

\begin{theorem}   \label{thm:b_p}
Let $B^d_p$ be the $d$-dimensional $\ell^p$-ball, with $d\geq 2$. For every $p\in (1,\infty)$ there
exists $b_p<3$ (computable from the above formulas) such that
\[
f(B^d_p)\geq (2^d+b_p^d)^{-1}
\]	 and \[
F(B^d_p)\geq (e+o(1))^{-1}(d\log d)^{-1}(2^d+b_p^d)^{-1}=(2^d+b_p^d)^{-d+o(d)}.
\] 
\end{theorem}

\subsection{An asymptotic lower bound}\label{sec:dvoretzky}

By the 2D case of Dvoretzky's theorem, any convex body of sufficiently high dimension admits an
approximately Euclidean $2$-dimensional central section (see~\cite[Chapter 5]{AGM15}). An easy
corollary is that Euclidean balls are asymptotic minimizers of $\rho_2(K)$ as $d\rightarrow +\infty$.  

\begin{proposition}
There exists a vanishingly small function $f(d)=o_{d\rightarrow \infty}(1)$ such that for all
symmetric convex bodies $K\subseteq \mathbb{R}^d$ we have $\rho_2(K)\geq \rho_2(B^d)-f(d)$. That is,
\[
\lim_{d\rightarrow \infty} \inf_{K\subseteq \mathbb{R}^d}\rho_2(K)=\sqrt{5+2\sqrt{3}}, 
\] 
where the infimum is over all symmetric convex bodies. 
\end{proposition}

\begin{proof}
Assume that $K$ is in John position, as we may since the constant $\rho_2(K)$ is
affine-invariant. Given $\epsilon>0$ there exists $d(\epsilon)$ such that any symmetric convex body
of dimension $d\geq d(\epsilon)$ has a central two-dimensional section that is approximately
Euclidean: there exists a $2$-dimensional subspace $H$ of $\mathbb{R}^d$ such that
\[ 
(1-\epsilon)\lVert x\rVert_2\leq \lVert x\rVert_K \leq \lVert x\rVert_2\qquad \forall x\in H. 
\]
Choose $x,y_1,y_2$ satisfying \eqref{eq:rho2_delta_inequalities} with respect to the Euclidean norm
and achieving the maximum in \eqref{eq:rho2_delta} as in Figure \ref{fig:rho2_euclidean}. Then
\[
\lVert x\rVert_K\leq 1 , \quad   
\lVert y_1\rVert_K,  \lVert y_2\rVert_K\leq 2, \quad \lVert y_1-y_2\rVert_K\geq 2\delta(1-\epsilon),
\] and
\[  
\min\{\lVert x-y_1\rVert_K, \lVert x-y_2\rVert_K\}\geq (1-\epsilon)\sqrt{5+2\sqrt{4-\delta^2}}.
\] Choose $\delta=\frac{1}{1-\epsilon}$: then $x,y_1,y_2$ satisfy the inequalities in the definition
of $\rho_2(K)$. Thus,
\[
\rho_2(K)\geq (1-\epsilon)\sqrt{5+2\sqrt{4-(1-\epsilon)^{-2}}}=\sqrt{5+2\sqrt{3}}-O(\epsilon), 
\]
which is the desired conclusion. 
\end{proof}

\section{Algorithmic version}\label{sec:alg}

In this section we prove Theorem~\ref{thm:alg}. 
Computing a maximum independent set in a finite collection of balls in $\RR^d$ is known to be
$\NP$-complete~\cite{GJ79}, even for $d=2$ and even if the balls are congruent. Here we show how
to bypass this difficulty using the polynomial time algorithm outlined below. This proves Theorem~\ref{thm:alg}.

\subsection{Preliminaries on volume computation} 

Computing the exact volume of the union of $n$ balls with arbitrary radii in $\mathbb{R}^d$ is a fundamental problem
in computational geometry and spatial analysis. While exact methods like power diagrams and $\alpha$-complexes are highly efficient
in low dimensions ($d \le 3$), they become intractable as $d$ grows due to \#P-hardness~\cite{BF10}. 

Let $\mathcal{B} = \{B_1, B_2, \dots, B_n\}$ be a set of $n$ balls in $\mathbb{R}^d$, where each ball $B_i$ is parameterized
by a center $c_i \in \mathbb{R}^d$ and an arbitrary radius $r_i > 0$. The objective is to calculate the total volume
$| \cup_{i=1}^n B_i|$.
Because exact computation is \#P-hard for high $d$, Fully Polynomial Randomized Approximation Schemes (FPRAS) are used, 
that guarantee a $(1 \pm \epsilon)$-approximation with high probability. Conveniently, such approximations suffice for our algorithm. 

We mention two methods. The classical Karp-Luby-Madras (KLM) sampling framework~\cite{KLM89} reduces
the problem of estimating a union to sampling from a disjoint universe (which is straightforward). In the context of geometric shapes,
it relies on three basic geometric primitives: computing the volume of a single ball,
uniformly sampling a point from a single ball, and checking point containment within a ball.
Another more recent approach~\cite{MAES18} replaces point sampling by ray sampling in order to
reduce the high variance in high dimensions due to the concentration of measure.
Both algorithms run in $O\left( n \cdot d \cdot \eps^{-2} \right)$ time and return  a $(1 \pm \epsilon)$-approximation 
of the volume with, \eg, probability $3/4$.

\subsection{Algorithms}

\noindent \textbf{Algorithm \texttt{A1} for unit balls.}
We say that an independent set $Y \subseteq X$  is \emph{stable with respect to one-for-two exchanges}
if $|Y|$ cannot increase by the removal of one of its elements and inclusion of two new elements of $X$,
while $Y$ remains an independent set.

  Let $|X|=n$.  For the algorithm, instead of having $Y$ as a maximum independent set, we relax this
  requirement to: Choose $Y \subseteq X$ as a subset subject to
  \[ |y -y'|>2 \qquad (y \neq y', y,y' \in Y), \]
  and such that $Y$ is (i)~\emph{maximal with respect to inclusion} and 
  (ii)~stable with respect to one-for-two exchanges.
  Note that the relaxed definition still works in the earlier argument for singly dominated centers.
  That is, if $Y$ satisfies these requirements, the set $A_y$ has diameter at most $4$ (cf.~Section \ref{sec:singly_dominated}). 

  The above requirements (i) and (ii) can be reached in time that is polynomial in $n$.
  Initially, $Y:= \emptyset$, and there are two mechanisms used for augmentation that preserve
  $Y$ as an independent set: add one element of $X \setminus Y$ or add two elements of $X \setminus Y$
  and remove one element of $Y$.

  The existence of a new element that can be added to $Y$  can be tested in $O(n^2)$ time:
  there are $n$ elements and any given element can be tested in $O(n)$ time with respect to
  its intersection pattern with the current set $Y$. 
  The existence of a one-for-two exchange can be tested in $O(n^3)$ time: there are
  ${n \choose 2}$ pairs of elements and any given pair can be tested in $O(n)$ time with respect to
  its intersection pattern with the current set $Y$. 
  Since the cardinality of $Y$ strictly increases by one after each augmentation,
  there are at most $n$ such augmentations possible.

  Recall that $\rho_2 = \rho_2(B^d) = \sqrt{5 + 2 \sqrt3} = 2.9093\ldots$ (cf. Proposition~\ref{prop:rho2_euclidean}).
  Set $\kappa:= (2^d + \rho_2^d)^{-1}$.
  The final set $Y$ satisfying
  \begin{equation} \label{eq:rho}
    \left| \cup_{y \in Y} B(y) \right| \geq \kappa \left| \cup_{x \in X} B(x) \right|,
  \end{equation}
  is output by the algorithm in $O(n^4)$ time.

\medskip
\noindent \textbf{Algorithm \texttt{A2} for arbitrary balls.}
We first recall some terminology from~\cite{D25}. Then we formulate an algorithmic version of the main result in~\cite{D25}.
Let $X=\mathbb{R}^d$ be $d$-dimensional Euclidean space.
Given a finite collection of balls $\mathcal{C}$ in $X$, define 
\[
\Phi(\C):= \max\left\{\frac{|\cup_{B\in \S}B|}{|\cup_{B\in \C}B|}: \S\subseteq \C \text{ is disjoint}\right\}.
\]
For a set $S \subset(0,\infty)$ of admissible radii, we denote by $\mathcal{A}(X; S)$ the set of all finite collections of balls in $X$ with radii in $S$.
Define
\[
\Gamma(X; S):=\inf\bigl\{\Phi(\mathcal{C}) : \mathcal{C}\in\mathcal{A}(X; S)\bigr\}
\] and \[
\Gamma(X):=\Gamma(X; (0,\infty)), \qquad \gamma(X):=\Gamma(X; \{1\}). 
\] 

The following properties of the constants $\Gamma(X; S)$ are easily verified:
\begin{enumerate} \itemsep 3pt
	\item if $S\subseteq T$ then $\Gamma(X; T)\leq \Gamma(X; S)$ (monotonicity), in particular $\Gamma(X)\leq \gamma(X)$; 
	\item $\Gamma(X; tS)=\Gamma(X; S)$ for all $t>0$ (scale-invariance); 
\end{enumerate} 

The main contribution in~\cite{D25} is the following theorem that relates $\Gamma(X)$ to $\gamma(X)$
(we still refer to $X=\RR^d$ throughout this section). 

\begin{theorem}[Approximate reduction to unit scale] \label{thm:many_vs_unit} {\rm \cite{D25}}
Let $L$ be a positive integer. Then 
\begin{equation}\label{eq:many_vs_unit}
		\Gamma(X)\geq (L+2)^{-1}\,\left((2L)^\frac{1}{L+1}(1+L^{-1})\right)^{-d}\, \gamma(X).
\end{equation} 
        In particular, choosing $L=d\log d+O(1)$, one gets
\begin{equation}\label{eq:many_vs_unit_optimized}
	\Gamma(X)\geq e^{-1-\frac{\log \log d}{\log d}+O\left(\frac{1}{\log d}\right)} \frac{\gamma(X)}{d\log d}. 
\end{equation} 
\end{theorem}

The proof of Theorem \ref{thm:many_vs_unit} is based on two lemmas for which we offer constructive versions.

\begin{lemma} \label{lem:window} {\rm \cite[Lemma~2.3]{D25}}
If $r>0$ and $\mu \geq 1$, then
\[	\Gamma\big(X; [r,\mu r]\big)\ \geq \mu^{-d}\gamma(X).	\] 
\end{lemma}

Its constructive version is as follows. Call the corresponding Algorithm \texttt{B1}.

\begin{lemma} \label{lem:window-con}
Let $\C$ be a collection of $n$ balls with radii in $[r, \mu r]$.
Then a disjoint sub-collection $\S$ whose volume is at least $\mu^{-d} \kappa |\cup_{B \in \C} B|$ can be
computed in $O(n^4)$ time.
\end{lemma}
\begin{proof}
Apply Algorithm \texttt{A1} on the set of concentric balls of radius $\mu r$ to compute a disjoint subset whose volume is guaranteed
in its analysis, cf.~\eqref{eq:rho}. The corresponding original balls are still disjoint and the volume loss $\mu^{-d}$ is accounted
in the text of the lemma. 
\end{proof} 

\medskip
For the second lemma, a definition is needed. 

\begin{definition}[(\(\lambda,\mu\))-lacunary sets of radii]
  Fix parameters $\lambda\geq 1$ and $\mu\geq 1$.
  A subset $S\subset(0,\infty)$ is \emph{$(\lambda,\mu)$-lacunary} if there exist intervals $[r_j,s_j]\subseteq (0,\infty)$ ($j=1,2,\ldots$) with
\[ r_{j+1}\ge \lambda\, s_j\quad\text{and}\quad s_j\le \mu r_j \qquad \forall j  \]
such that $S\subset \bigcup_{j\ge1}[r_j,s_j]$.
\end{definition}

\begin{lemma} \label{lem:lacunary} {\rm \cite[Lemma~2.5]{D25}}
Let $S\subset(0,\infty)$ be $(\lambda,\mu)$-lacunary, where $\lambda> 1$ and $\mu\geq 1$. Then
\[ \Gamma(X; S)\ \geq \mu^{-d}\, \bigl(1+2\lambda^{-1}\bigr)^{-d}\,  \gamma(X). \]
\end{lemma}

Its constructive version is as follows. Call the corresponding Algorithm \texttt{B2}.

\begin{lemma} \label{lem:lacunary-con}
Let $S\subset(0,\infty)$ be $(\lambda,\mu)$-lacunary, where $\lambda> 1$ and $\mu\geq 1$.
Let $\C$ be a finite collection of balls with radii in $S$.
Then a disjoint sub-collection $\S$ whose volume is at least
$\mu^{-d} \, \bigl(1+2\lambda^{-1}\bigr)^{-d}\, \kappa |\cup_{B \in \C} B|$ can be
computed in $O(n^4)$ time.
\end{lemma}
\begin{proof}
Let $\C$ be any finite collection of balls with radii in $S$. By $(\lambda,\mu)$-lacunarity, $S\subset\bigcup_{j\geq 1} [r_j,s_j]$
with $s_j\leq \mu r_j$ and $r_{j+1}\geq \lambda s_j$. 
Since $\C$ is finite, there exists a positive integer $m$ such that $S$ intersects $[r_j,s_j]$ only for $j\leq m$.
Let $\C_j$ be the sub-collection of $\C$ consisting of balls with radii in $[r_j,s_j]$, for $j=1,\ldots,m$. 
First, a greedy algorithm is used to construct  $m$ pairwise disjoint collections $\C'_j$, for $j=1,\ldots,m$,  as outlined below.

Set  $\C'_m := \C_m$, and for $j=m-1,\dots,1$, let $\C'_j$ be the set of balls
in $\C_j$ that do not meet any ball in $\bigcup_{k>j}\C'_k$. Set $\C':=\bigcup_{j=1}^m \C'_j$.
Then by construction, any two balls from different $\C'_j$ are disjoint. 

Construct the collection obtained by enlarging each ball in $\C'$ by a factor of $1+2\lambda^{-1}$.  Then
(details in~\cite[Proof~of~Lem~2.5]{D25}) the union of these enlarged balls contains the union of the balls in $\C$, namely
\[ \cup_{B\in \C} B\ \subseteq\ \cup_{B'\in\C'} (1+2\lambda^{-1})\,B'. \]

For each $j=1,\ldots,m$, apply  Algorithm \texttt{B1} to $\C''_j:=\{(1+2\lambda^{-1})B'\colon\, B'\in\C'_j\}$, 
which is a collection of enlarged balls with radii in
\[ [(1+2\lambda^{-1})r_j,(1+2\lambda^{-1})s_j] \subset [(1+2\lambda^{-1})r_j,\, \mu  (1+2\lambda^{-1})r_j]. \]

By Lemma~\ref{lem:window} for each $j=1,\ldots,m$, Algorithm \texttt{B1} gives a disjoint sub-collection $\S''_j \subseteq \C''_j$ with
\begin{equation}\label{eq:S''} 
|\cup_{B''\in \S''_j} B''|\geq \mu^{-d} \kappa \, |\cup_{B''\in \C''_j} B''|.
\end{equation} 

Summing over $j$, the sum of the volumes of the $m$ disjoint sub-collections $S''_j$ retains a large fraction
of the volume of $\C$ (details in~\cite[Proof~of~Lem~2.5]{D25}), namely:
\begin{equation}
  \sum_{j=1}^m|\cup_{B''\in \S''_j} B''| \geq \mu^{-d} \kappa \, |\cup_{B\in \C} B|.
\end{equation}

Compute the sub-collection $\S'_j=\{B'\in \C'_j\colon\, (1+2\lambda^{-1}) B'\in \S_j''\}$, which is clearly disjoint (since $\S''_j$ is disjoint). 
Compute the union of these sub-collections, namely the collection $\S'=\cup_{j=1}^m\S'_j$. The collection $\S'$  is also disjoint.
Further, the volume loss is expressed below:
\begin{equation} \label{eq:S'} 
|\cup_{B'\in \S'} B'| = \sum_{j=1}^m|\cup_{B'\in \S'_j} B'|= (1+2\lambda^{-1})^{-d}\sum_{j=1}^m|\cup_{B''\in \S''_j} B''|, 
\end{equation}
where the disjointness of each $\mathcal{S}''_j$ has been used.

Overall, $\S'$ is a disjoint sub-collection of $\C$, whose volume is at least $\mu^{-d}\, (1+2\lambda^{-1})^{-d} \kappa $
times the volume of $\C$, namely
\[ | \cup_{B'\in \S'} B'|  \geq  \mu^{-d}\, (1+2\lambda^{-1})^{-d} \kappa \, |\cup_{B\in \C} B|, \]
as required.

There are two sources of loss in the volume, accounted for by $(\lambda,\mu)$-lacunarity and quantified in the text of the lemma:
$\mu^{-d}$ from~\eqref{eq:S''}, and $(1+2\lambda^{-1})^{-d}$ from~\eqref{eq:S'}. The total loss is bounded by the product of these factors.
Algorithm \texttt{B1} is applied $m$ times: if $\C''_j$ has cardinality $n_j$, then the running time is
$O \left( \sum_j n_j^4 \right ) =O \left ((\sum_j n_j)^4\right ) = O \left (n^4 \right )$.
\end{proof}

\medskip
Theorem \ref{thm:many_vs_unit} follows by combining Lemmas~\ref{lem:window} and~\ref{lem:lacunary} with a pigeonhole argument.
Specifically Algorithm \texttt{B2} is used.
The variable~$L$ in the statement of Theorem \ref{thm:many_vs_unit} is the number of pigeonholes used.  
The resulting algorithm is as follows:

After a suitable scaling, if needed, we may assume that the minimum ball radius is $1$.
Set $L := \lceil d \log {d} \rceil$, and $J:= L+2$; and $\lambda:= (2L)^{1/(L+1)}$;
Note that for $d \geq 3$, we have $L \geq 4$ and $J \geq 6$, thus $\lambda>1$.
(This parameter setting is made according to~\cite{D25}.)

For each residue class $i\in\{0,1,\dots,J-1\}$ define
\[
S_i := \bigcup_{k\in\mathbb{N}} \bigl[\lambda^{kJ+i},\ \lambda^{kJ+i+1}\bigr].
\]
Observe that ror each $i$, the set $S_i$ is $(\lambda^{J-1},\lambda)$-lacunary.
Clearly, $\bigcup_{i=0}^{J-1}S_i=[1,\infty)$. Let $\C$ be any finite collection of balls. Partition it into the sub-collections
\[\C_i:=\{B\in\C: \ B \ \text{has radius in } S_i\}.\]  
Because $|\cup_{B\in \mathcal{C}}B|\leq \sum_{i=0}^{J-1}|\cup_{B\in \mathcal{C}_i}B|$, by the pigeonhole principle, there exists $i_\ast$ with
\[ |\cup_{B\in \mathcal{C}_{i_\ast}}B|\geq J^{-1}|\cup_{B\in \mathcal{C}}B|.  \]

Constructively, the algorithm retains the approximately heaviest subsequence from the $J$ available,
\ie, the subsequence whose sub-collection of balls has the largest volume of the union, up to the $(1 \pm \eps)$-approximation.

Since $S_i$ is $(\lambda^{J-1},\lambda)$-lacunary, applying  Algorithm \texttt{B2}
to $\C_{i_\ast}$ yields a disjoint sub-family $\mathcal{S}\subseteq\mathcal{C}_{i_\ast}$ with
\[
|\cup_{B\in \S} B |\ge\ \lambda^{-d}\,\bigl(1+2\lambda^{1-J}\bigr)^{-d} \kappa \, |\cup_{B\in \C_{i_\ast}} B|
\ \ge\ J^{-1}\,\lambda^{-d}\,\bigl(1+2\lambda^{1-J}\bigr)^{-d} \kappa \, |\cup_{B\in \C}B|.
\]
The parameters $J,L,\lambda,\mu$ have been chosen so that the losses incurred by the application of the pigeonhole principle
and  Algorithm \texttt{B2} are absorbed, for large $d$, by any slight increase of the exponential basis  $\sqrt{5+2\sqrt{3}}$
in Proposition~\ref{prop:rho2_euclidean}; \eg, this value can be replaced by $2.910$. 
Finally, the $(1 \pm \eps)$-approximation factor in computing the volume of a union of balls is absorbed in a similar way. 
Consequently, Algorithm \texttt{A2} computes a disjoint sub-collection of balls $\S \subseteq \C$ whose volume is at least
\begin{equation*}
|\cup_{B\in \mathcal{S}}B|  \geq 2.910^{-d} |\cup_{B\in \mathcal{C}}B|,
\end{equation*}
for large $d$. Choosing, say, $\eps=0.1$, allows the pigeonholing to run in $O(\sum_1^J d n_i) = O(d n)$ time, 
where $\sum_1^J n_i =n$. Altogether, this procedure and the  greedy algorithm run in $O(d n + n^4)$ time,
concluding the proof of Theorem~\ref{thm:alg}. Note that under the most common scenario, $d = O(n^3)$, thus
$O(d n + n^4) =O(n^4)$.

\section{Upper bound} \label{sec:upper}

In this section we study $F(B^d)$ and $f(B^d)$ in the regime of high dimensions and prove Theorem~\ref{thm:upper}.
We use the celebrated Kabatiansky--Levenshtein upper bound for spherical codes (see~\cite{KL89}) to obtain an exponential
strengthening of the upper bound $f(B^d)\leq 2^{-d}$, highlighting in particular that the behavior of the
constants $f(B^d)$ and $F(B^d)$ is quite different from their analogues for cubes. We prove~\eqref{eq:b>2.44}, namely
that $f(B^d) = O(2.447^{-d})$.

A weaker  upper bound obtained by more elementary means, $f(B^d) = O(2.414^{-d})$,
can be found in~\cite{DD26av3}.  It is the precursor of the bound in Theorem~\ref{thm:upper}.

We now prove Theorem~\ref{thm:upper}, making use of the best known linear programming bound for
spherical codes (see~\cite{KL89, CZ14}), which allows us to look at a scale slightly larger
than $\sqrt{2}$ (cf.~\cite[Thm~6.2]{DD26av3}). We recall that a \emph{spherical code} in dimension $d$
with minimum angle $\theta$ is a set of points on the unit sphere of $\RR^d$ such that no two points
subtend an angle less than $\theta$ at the origin. We denote by $A(d,\theta)$ the maximal cardinality
of such a spherical code.  

\begin{proof}[Proof of Theorem \ref{thm:upper}]
  Let $r\in [1,2]$ and consider the collection $\mathcal{C}_r^d$ of all unit balls in $\RR^d$ with
  center in $B^d(r)=\{x\in \RR^d\colon\, |x|\leq r\}$. 
  If $N_r(d)$ is the maximal cardinality
  of a disjoint sub-collection of $\mathcal{C}^d_r$, then \begin{equation}\label{eq:ball_r}
f(B^d)\leq \frac{N_r(d)}{(r+1)^d}. 
\end{equation}
  There are (at least) $N_r(d)-1$ points $x_1,\ldots, x_{N_r(d)-1}\in B_d(0,r)\setminus\{0\}$, whose
  mutual distance is at least $2$. 
  Projecting them radially 
  onto the unit sphere, that is, considering  $y_j=\frac{x_j}{|x_j|}$ ($j=1,\ldots, N_r(d)-1$), we
  get a spherical code of angle $\theta$, 
  where $\theta\in [\pi/3, \pi]$ is uniquely defined by $\frac{1}{r}=\sin(\theta/2)$.
  This follows from Lemma~2.2 of~\cite{CZ14}. Thus,
\[ N_r(d)\leq A(d,\theta)+1\leq 2A(d,\theta). \] 
Combining this upper bound with~\eqref{eq:ball_r}, we get
\begin{eqnarray*}
  \limsup_{d\rightarrow \infty}\frac{1}{d}\log f(B^d)&\leq&
  \limsup_{d\rightarrow \infty}\frac{1}{d}\log A(d,\theta) -\log(r+1)\\
&\leq& \limsup_{d\rightarrow \infty}\frac{1}{d}\log A(d,\theta) -\log(1+1/\sin(\theta/2)).
\end{eqnarray*}
It has been shown by Kabatiansky--Levenshtein~\cite{KL89} (see also formula (2.6) of~\cite{CZ14}) that
\[ \frac{1}{d}\log A(d,\theta) \leq (1+o_{d\rightarrow \infty}(1)) \left(\frac{1+\sin\theta}{2\sin\theta}
\log\frac{1+\sin\theta}{2\sin\theta}-\frac{1-\sin\theta}{2\sin\theta} \log\frac{1-\sin\theta}{2\sin\theta}\right). \]
Thus, $\limsup_{d\rightarrow \infty}\frac{1}{d}\log f(B^d)$ is at most
\[
\left(\frac{1+\sin\theta}{2\sin\theta} \log\frac{1+\sin\theta}{2\sin\theta}-\frac{1-\sin\theta}
     {2\sin\theta} \log\frac{1-\sin\theta}{2\sin\theta}\right)-\log(1+1/\sin(\theta/2)),
\]
where $\theta$ is any angle in $[\pi/3,\pi]$ (notice that any such angle is achieved by an
appropriate choice
of $r\in [1,2]$). The numerical value of this minimum\footnote{The authors used Wolfram Alpha for this
  computation.} 
is $-0.895227...$.
Since $\exp(0.895227...)=2.44789...$ we deduce that
\[ f(B^d)\leq (2.44789)^{-d+o(d)}.  \] 
The minimum is achieved at $\theta=1.45251...$, that is, $r=1.50586...>\sqrt{2}$.
Thus, our upper bound is obtained by looking at
all unit balls contained in a ball of radius $2.50586...$.
\end{proof}

\section{Concluding Remarks}\label{sec:concluding}

\begin{enumerate} \itemsep 3pt
  
\item There are now two very different sources of exponentially improved lower bounds on
          $f(K)$. If $K$ is a cube, or more generally the Voronoi cell of a lattice (an observation
          of Rado~\cite{R49}, see also~\cite[Section 4]{DD26av3}), then $f(K)=2^{-d}$ and the
          exponential gap between upper and lower bounds is in fact closed; if $K$ is sufficiently
          convex in the precise sense that $\delta_K(1/2)\geq c_0>0$, then we showed above that the
          exponential gap can be narrowed. The two proofs exploit opposite features of
          $K$: a tiling property in the first case, and strict convexity in the second.
          
\item The proof scheme of Section \ref{sec:general} lends itself to a natural refinement:
          decompose the set of centers as
          $X=X_1\cup X_2\cup X_{\geq 3}$, where \[ X_2:= \{ x \in X \ \colon \ \# N(x)=2\}, \qquad
	X_{\geq 3}:= \{ x \in X \ \colon \ \# N(x) \geq 3\}, \]
	or $X=X_1\cup X_2\cup X_3\cup X_{\geq 4}$ (with obvious interpretation of the notation),
        etc. One could define a constant $\rho_m(K)$ for every $m\geq 2$, as the minimal dilation
        factor $r$ with the property that whenever $m$ interior-disjoint translates of $K$ meet
        another such translate, then the latter is contained in the union of the $r$-dilates of the
        former. This leads to an improved estimate for the contribution of $X_{\geq m}$ to the
        total volume, via an appropriate analogue of Proposition \ref{prop:rho2_vs_convexity}. The
        most difficult part of such an analysis is controlling the additional sets $X_2,X_3,\ldots$.
        Are there analogues of the isodiametric inequality that could take care of that?
        Candidate conjectural statements of this kind can be found in~\cite{DHK+21}.

\item An LLM, prompted by the first author, succeeded in executing the plan in the point
          above for $m=3$. A write-up clarifying the argument and tracing possible antecedents in
          the literature is in preparation.
          
\item In some sense, closing the remaining exponential gap
	$2.910^{-d} \leq f(B^d) \leq 2.447^{-d} $
	in Rado's covering problem in high dimensions,
	remains as challenging as it was $80$ years ago.
        
\end{enumerate}

\end{document}